\documentclass[11pt,reqno]{amsart}
\usepackage[breaklinks]{hyperref}
\usepackage{url}
\usepackage{amssymb}

\newcommand{\G}{\Gamma}

\renewcommand{\(}{\left\(}
\renewcommand{\)}{\right\)}
\renewcommand{\[}{\left\[}
\renewcommand{\]}{\right\]}

\numberwithin{equation}{section}
 \theoremstyle{plain}
\newtheorem{theorem}{Theorem}[section]
\newtheorem{lemma}[theorem]{Lemma}
\newtheorem{corollary}[theorem]{Corollary}

   \makeatletter
\def\proof{\@ifnextchar[{\@oproof}{\@nproof}}
\def\@oproof[#1][#2]{\trivlist\item[\hskip\labelsep\textit{#2 Proof of\
#1.}~]\ignorespaces}
\def\@nproof{\trivlist\item[\hskip\labelsep\textit{Proof.}~]\ignorespaces}

\makeatother

\newmuskip\pFqmuskip

\newcommand*\pFq[6][8]{%
  \begingroup 
  \pFqmuskip=#1mu\relax
  \mathchardef\normalcomma=\mathcode`,
  \mathcode`\,=\string"8000
  \begingroup\lccode`\~=`\,
  \lowercase{\endgroup\let~}\pFqcomma
  {}_{#2}F_{#3}{\left[\genfrac..{0pt}{}{#4}{#5};#6\right]}%
  \endgroup
}
\newcommand{\pFqcomma}{{\normalcomma}\mskip\pFqmuskip}

\begin{document}
\title[Extensions of Watson's theorem and the Ramanujan-Guinand formula]{Extensions of Watson's theorem and the Ramanujan-Guinand formula}
\author{Rahul Kumar}
\address{ Discipline of Mathematics, Indian Institute of Technology Gandhinagar, Palaj, Gandhinagar 382355, Gujarat, India}\email{rahul.kumr@iitgn.ac.in}

\maketitle

\begin{abstract}
Ramanujan provided several results involving the modified Bessel function $K_z(x)$ in his Lost Notebook. One of them is the famous Ramanujan-Guinand formula, equivalent to the functional equation of the non-holomorphic Eiesenstien series on $SL_2(z)$. Recently, this formula was generalized by Dixit, Kesarwani, and Moll. In this article, we first obtain a generalization of a theorem of Watson and, as an application of it, give a new proof of the result of Dixit, Kesarwani, and Moll. Watson's theorem is also generalized in a different direction using ${}_\mu K_z(x,\lambda)$ which is itself a generalization of $K_z(x)$. Analytic continuation of all these results are also given.
\end{abstract}

\section{Introduction}

On page 253 of his Lost Notebook \cite{rlnb}, Ramanujan recorded the following beautiful formula. 
Let $z\in\mathbb{C}$. For $\alpha,\ \beta>0$ such that $\alpha\beta=\pi^2$,
\begin{align}\label{ramgui}
&\sqrt{\alpha}\sum_{n=1}^\infty \sigma_{-z}(n)n^{\frac{z}{2}}K_{\frac{z}{2}}(2n\alpha)-\sqrt{\beta}\sum_{n=1}^\infty\sigma_{-z}(n)n^{\frac{z}{2}}K_{\frac{z}{2}}(2n\beta)\nonumber\\
&=\frac{1}{4}\Gamma\left(\frac{z}{2}\right)\zeta(z)\left\{\beta^{\frac{1-z}{2}}-\alpha^{\frac{1-z}{2}}\right\}+\frac{1}{4}\Gamma\left(-\frac{z}{2}\right)\zeta(-z)\left\{\beta^{\frac{1+z}{2}}-\alpha^{\frac{1+z}{2}}\right\},
\end{align}
where $K_z(x)$ is the modified Bessel function \cite[p.~34]{watson} and $\sigma_z(n):=\sum_{d|n}d^z$ is the generalized divisor function. 

Theorem \ref{ramgui} was rediscovered by Guinand \cite{guinand} in 1955 which is why sometimes this formula is referred to as Ramanujan-Guinand formula in the literature. Several years after Ramanujan, the special case $z=0$ of \eqref{ramgui} was also rediscovered by Koshliakov \cite{kosh}.  For a detailed history of the Ramanujan-Guinand formula \eqref{ramgui} and its implications, see \cite{bls}. 

It is well-known that \eqref{ramgui} is equivalent to the functional equation of the non-holomorphic Eisenstein series. For Re$(s)>1$ and $\tau:=x+iy,\ y>0$,  the non-holomorphic Eisenstein series $G(\tau,s)$ is defined by  
\begin{align*}
G(\tau,s):=\frac{1}{2}\sum_{\substack{{(m,n)\in\mathbb{Z}^2}\\{(m,n)\neq(0,0)}}}\frac{y^s}{|m\tau+n|^{2s}}.
\end{align*}
$G(\tau,s)$ is a non-holomorphic modular function of weight $0$ as it is invariant in $\tau$ under the usual action of SL$_2(\mathbb{Z})$. Moreover, $G(\tau,s)$ has analytic continuation in $\mathbb{C}$ except a simple pole at $s=1$ with residue $\frac{\pi}{2}$. Analytic continuation of this function can be given by its Fourier expansion
\begin{align}\label{fourier}
G(\tau,s)=\zeta(2s)y^s+\frac{\sqrt{\pi}\Gamma\left(s-\frac{1}{2}\right)}{\Gamma(s)}\zeta(2s-1)y^{1-s}+\frac{4\sqrt{y}\pi^s}{\Gamma(s)}\sum_{n=1}^\infty\sigma_{2s-1}(n)n^{\frac{1}{2}-s}K_{s-\frac{1}{2}}(2\pi ny)\cos(2\pi nx).
\end{align}
It is effortless to see that \eqref{ramgui} follows from \eqref{fourier} and by using the modular property of $G(\tau,s)$, namely, $G(\tau,s)=G\left(-\frac{1}{\tau},s\right),\ \tau\in\mathbb{H}$(upper half-plane).

Several proofs of \eqref{ramgui} have been given in the literature, for example, see \cite{bls}, \cite{dixitmontash}, and \cite{guinand}. The main ingredient in the proofs of \cite{bls} and \cite{guinand} is the following result of Watson \cite{watsonself}:
\begin{theorem}\label{watson}
For Re$(x)>0$ and Re$(z)>0$, we have
\begin{align}\label{watsonlemma}
&2\sum_{n=1}^\infty \left(\frac{nx}{2}\right)^z K_z(n x)+\frac{1}{2}\Gamma(z)-\frac{\sqrt{\pi}}{x}\Gamma\left(z+\frac{1}{2}\right)=2\sqrt{\pi}x^{2z}\Gamma\left(z+\frac{1}{2}\right)\sum_{n=1}^\infty \frac{1}{(x^2+4n^2\pi^2)^{z+\frac{1}{2}}}.
\end{align}
\end{theorem}

In \cite{dixitmontash}, Dixit provided another proof of the Ramanujan-Guinand formula \eqref{ramgui} without invoking Theorem \ref{watson}. For the details of the proof, we refer the reader to \cite[Theorem 1.4]{dixitmontash}.

The main goal of this paper is to accomplish two different generalizations of Theorem \ref{watson} along with an application of one of the theorem. In the literature, Theorem \ref{watson} has been generalized in many directions, for example, see \cite{Berndt}, \cite{kober}. Theorem \ref{watson} has important application in Number Theory, for example, see \cite{bls} and the references therein. The series on the left-hand side of \eqref{ramgui} has played an important role in the work of Gupta and the author  \cite{guptakumar}. Moreover, very recently, Theorem \ref{watson} is employed nicely in \cite[Section 9]{dksuper} to obtain a modular relation involving the generalized Hurwitz zeta function \cite[Equation~(1.3)]{dksuper}
\begin{align*}
\zeta_w(s, a)&:=\frac{4}{w^2\sqrt{\pi}\G\left(\frac{s+1}{2}\right)}\sum_{n=1}^{\infty}\int_{0}^{\infty}\int_{0}^{\infty}\frac{(uv)^{s-1}e^{-(u^2+v^2)}\sin(wv)\sinh(wu)}{\left(n^2u^2+(a-1)^2v^2\right)^{s/2}}\, dudv,
\end{align*}
where $w\in\mathbb{C}\backslash\{0\}$, Re$(s)>1$ and $a\in\mathfrak{B}:=\{\xi:\textup{Re}(\xi)=1, \textup{Im}(\xi)\neq 0\}$. For the theory of $\zeta_w(s, a)$, we refer the reader to \cite{dksuper}.

Recently, Dixit, Kesarwani, and Moll \cite{dkmm} provided an elegant generalization of Theorem \ref{ramgui}. To state their result, we first need to define their new generalization of the modified Bessel function $K_z(x)$. For $z,w\in\mathbb{C}$, and $ x\in\mathbb{C}\backslash\{x\in\mathbb{R}:x\leq0\}$, the generalized modified Bessel function is defined by  \cite[Equation (1.3)]{dkmm}
\begin{align}\label{kzwdef}
K_{z,w}(x):&=\frac{1}{2\pi i}\int_{(c)}\Gamma\left(\frac{s-z}{2}\right)\Gamma\left(\frac{s+z}{2}\right){}_1F_1\left(\frac{s-z}{2};\frac{1}{2};\frac{-w^2}{4}\right){}_1F_1\left(\frac{s+z}{2};\frac{1}{2};\frac{-w^2}{4}\right) 2^{s-2}x^{-s}\ ds,
\end{align}
where $c:=\mathrm{Re}(s)>\left|\mathrm{Re}(z)\right|$ and ${}_1F_1(a;c;z)$ is the confluent hypergeometric function defined by \cite[p.~172, Equation (7.3)]{temme}
\begin{align*}
{}_1F_1(a;c;z):=\sum_{n=0}^\infty\frac{(a)_n}{(c)_n}\frac{z^n}{n!},\ |z|<\infty,
\end{align*}
with $(a)_m:=\Gamma(a+m)/\Gamma(a)$ for $a\in\mathbb{C}$. Here, and throughout the paper, $\int_{(c)}$ denotes the line integral $\int_{c-i\infty}^{c+i\infty}$.

It is straightforward to see that for $w=0$, $K_{z,w}(x)$ reduced to $K_{z}(x)$ by using \cite[p.~115, Formula~11.1]{ober}
\begin{align*}
K_z(x)=\frac{1}{2\pi i}\int_{(c)}\Gamma\left(\frac{s-z}{2}\right)\Gamma\left(\frac{s+z}{2}\right)2^{s-2}x^{-s}\ ds.
\end{align*}
For more details on the theory of $K_{z,w}(x)$ and its connection in Analytic Number Theory and Physics, we refer the reader to \cite{dkmm} and \cite{rk}.

The aforementioned generalization of the Ramanujan-Guinand formula is given in the following theorem \cite[Theorem 1.4]{dkmm}.
\begin{theorem}\label{dkm}
Let $z, w\in\mathbb{C}$. Let $K_{z,w}(x)$ be defined in \eqref{kzwdef}. For $\alpha, \beta>0$ such that $\alpha\beta=\pi^2$,
\begin{align}\label{dkmeqn}
&\sqrt{\alpha}\sum_{n=1}^\infty \sigma_{-z}(n)n^{\frac{z}{2}}e^{-\frac{w^2}{4}}K_{\frac{z}{2},iw}(2n\alpha)-\sqrt{\beta}\sum_{n=1}^\infty\sigma_{-z}(n)n^{\frac{z}{2}}e^{\frac{w^2}{4}}K_{\frac{z}{2},w}(2n\beta)\nonumber\\
&=\frac{1}{4}\Gamma\left(\frac{z}{2}\right)\zeta(z)\left\{\beta^{\frac{1-z}{2}}{}_1F_1\left(\frac{1-z}{2};\frac{1}{2};\frac{w^2}{4}\right)-\alpha^{\frac{1-z}{2}}{}_1F_1\left(\frac{1-z}{2};\frac{1}{2};-\frac{w^2}{4}\right)\right\}\nonumber\\
&\quad+\frac{1}{4}\Gamma\left(-\frac{z}{2}\right)\zeta(-z)\left\{\beta^{\frac{1+z}{2}}{}_1F_1\left(\frac{1+z}{2};\frac{1}{2};\frac{w^2}{4}\right)-\alpha^{\frac{1+z}{2}}{}_1F_1\left(\frac{1+z}{2};\frac{1}{2};-\frac{w^2}{4}\right)\right\}.
\end{align}
\end{theorem}
Here we note that authors of \cite{dkmm} did not follow the approach used in \cite{bls} or \cite{guinand} to prove Theorem \ref{dkm}. They used the theory of functions reciprocal in a certain kernel.  To follow the same approach as in \cite{bls} or \cite{guinand}, one needs to first obtain generalization of the Watson's result \eqref{watsonlemma} which has been missing from the literature. In this work, we fill this gap, that is, we obtain a generalization of the Watson's result \eqref{watsonlemma}. Further as an application of it, we provide a new proof of Theorem \ref{dkm}. Our generalization of \eqref{watsonlemma} is recorded in the following theorem.
\begin{theorem}\label{withkzw}
Let $\mathrm{Re}(x^2)>0$. Define
\begin{align}\label{Ax}
\mathcal{A}(n,z,w,x):={}_1F_1\left(\frac{1}{2}+z;\frac{1}{2};\frac{w^2(2n\pi+ix)}{8n\pi-4ix}\right)+{}_1F_1\left(\frac{1}{2}+z;\frac{1}{2};\frac{w^2(2n\pi-ix)}{8n\pi+4ix}\right).
\end{align}
Let $w\in\mathbb{C}$ and $\mathrm{Re}(z)>0$. Let $K_{z,w}(x)$ be defined in \eqref{kzwdef}. Then
\begin{align}\label{withkzweqn}
&2\sum_{n=1}^\infty\left(\frac{nx}{2}\right)^zK_{z,w}(nx)+\frac{1}{2}\Gamma(z){}_1F_1\left(z;\frac{1}{2};-\frac{w^2}{4}\right)-\frac{\sqrt{\pi}}{x}\Gamma\left(\frac{1}{2}+z\right)e^{-\frac{w^2}{4}}{}_1F_1\left(\frac{1}{2}+z;\frac{1}{2};-\frac{w^2}{4}\right)\nonumber\\
&=\sqrt{\pi}x^{2z}\Gamma\left(\frac{1}{2}+z\right)e^{-\frac{w^2}{4}}\sum_{n=1}^\infty\frac{\mathcal{A}(n,z,w,x)}{(x^2+4n^2\pi^2)^{z+\frac{1}{2}}}.
\end{align}
\end{theorem}

We now obtain an extended version of Theorem \ref{withkzw} in which the restriction on $z$ is removed.
\begin{theorem}\label{continuationwithkzw}
Let $w\in\mathbb{C}$, $\mathrm{Re}(x^2)>0$, and $\mathcal{A}(n,z,w,x)$ be defined in \eqref{Ax}. Let $M>0$ be an integer. Then for $\mathrm{Re}(z)>-M$, 
{\allowdisplaybreaks\begin{align}\label{continuationwithkzweqn}
&2\sum_{n=1}^\infty\left(\frac{nx}{2}\right)^zK_{z,w}(nx)+\frac{1}{2}\Gamma(z){}_1F_1\left(z;\frac{1}{2};-\frac{w^2}{4}\right)-\frac{\sqrt{\pi}}{x}\Gamma\left(\frac{1}{2}+z\right)e^{-\frac{w^2}{4}}{}_1F_1\left(\frac{1}{2}+z;\frac{1}{2};-\frac{w^2}{4}\right)\nonumber\\
&=\sqrt{\pi}x^{2z}\Gamma\left(\frac{1}{2}+z\right)e^{-\frac{w^2}{4}}\Bigg[\sum_{n=1}^\infty \mathcal{A}(n,z,w,x)\Bigg\{\frac{1}{(x^2+4n^2\pi^2)^{z+\frac{1}{2}}}-\sum_{m=0}^{M-1}\binom{-z-\frac{1}{2}}{m}\frac{x^{2m}}{(2n\pi)^{2z+2m+1}}\Bigg\}\nonumber\\
&\qquad+\sum_{m=0}^{M-1}\binom{-z-\frac{1}{2}}{m}\frac{x^{2m}}{(2\pi)^{2z+2m+1}}\sum_{n=1}^\infty \frac{\mathcal{A}(n,z,w,x)}{n^{2z+2m+1}}\Bigg].
\end{align}}
\end{theorem}

As a special case of Theorem \ref{continuationwithkzw}, we get the result of Watson \cite[p.~300, Section~3]{watsonself}:
\begin{corollary}\label{spwatson}
Let $\mathrm{Re}(x^2)>0$. Let $M>0$ be an integer. Then for $\mathrm{Re}(z)>-M$, 
{\allowdisplaybreaks\begin{align}\label{spwatsoneqn}
&2\sum_{n=1}^\infty\left(\frac{nx}{2}\right)^zK_{z}(nx)+\frac{1}{2}\Gamma(z)-\frac{\sqrt{\pi}}{x}\Gamma\left(\frac{1}{2}+z\right)\nonumber\\
&=2\sqrt{\pi}x^{2z}\Gamma\left(\frac{1}{2}+z\right)\Bigg[\sum_{n=1}^\infty \Bigg\{\frac{1}{(x^2+4n^2\pi^2)^{z+\frac{1}{2}}}-\sum_{m=0}^{M-1}\binom{-z-\frac{1}{2}}{m}\frac{x^{2m}}{(2n\pi)^{2z+2m+1}}\Bigg\}\nonumber\\
&\qquad+\sum_{m=0}^{M-1}\binom{-z-\frac{1}{2}}{m}\frac{x^{2m}\zeta(2z+2m+1)}{(2\pi)^{2z+2m+1}}\Bigg].
\end{align}}
\end{corollary}

Next, we present a generalization of Theorem \ref{watson} in another direction, different from Theorem \ref{withkzw}. Very recently, a new two-parameter generalization of $K_z(x)$ was introduced by Dixit, Kesarwani and the author in \cite{dkk}. For $z\in\mathbb{C}\backslash\left(\mathbb{Z}\backslash\{0\}\right)$, and $x, \mu,\lambda\in\mathbb{C}$ such that $\mu+\lambda\neq-\frac{1}{2}, -\frac{3}{2}, -\frac{5}{2},\cdots$, by \cite[Equation (1.16)]{dkk}, it is given by
\begin{align}\label{muknuwdef}
{}_{\mu}K_{z}(x, \lambda)&:=\frac{\pi x^\lambda 2^{\mu+z-1}}{\sin(z\pi)}\Bigg\{\left(\frac{x}{2}\right)^{-z}\frac{\Gamma(\mu+\lambda+\tfrac{1}{2})}{\Gamma(1-z)\Gamma(\lambda+\tfrac{1}{2}-z)}\pFq12{\mu+\lambda+\tfrac{1}{2}}{\lambda+\tfrac{1}{2}-z,1-z}{\frac{x^2}{4}}\nonumber\\
&\quad\quad\quad\quad\quad\quad-\left(\frac{x}{2}\right)^{z}\frac{\Gamma(\mu+z+\lambda+\tfrac{1}{2})}{\Gamma(1+z)\Gamma(\lambda+\tfrac{1}{2})}\pFq12{\mu+z+\lambda+\tfrac{1}{2}}{\lambda+\tfrac{1}{2},1+z}{\frac{x^2}{4}}\Bigg\},
\end{align}
with ${}_{\mu}K_{0}(x, \lambda)=\lim_{z\to0}{}_{\mu}K_{z}(x, \lambda)$.

From \cite[equation (1.17)]{dkk}, we have 
\begin{equation*}
{}_{-z}K_{z}(x, \lambda)=x^{\lambda}K_{z}(x).
\end{equation*}
Therefore, it is obvious to see that if $\mu=-z$ and $\lambda=0$ then ${}_{\mu}K_{z}(x, \lambda)$ reduces to $K_z(x)$. For more information on ${}_{\mu}K_{z}(x, \lambda)$ and its number theoretic applications, we refer the reader to \cite{dkk}.

Our second generalization of the Watson's result \eqref{watsonlemma} is given below.
\begin{theorem}\label{withmuknu}
Let $\mathrm{Re}(x)>0$, $\mathrm{Re}(z)>0$ and $\mathrm{Re}(\mu+\lambda)>0$. Let ${}_\mu K_{z}(x,\lambda)$ be defined in \eqref{muknuwdef}. Then 
\begin{align}\label{withmuknueqn}
&2\sum_{n=1}^\infty\left(\frac{nx}{2}\right)^{z-\lambda}{}_{\mu}K_{z}(nx,\lambda)+\frac{2^{z+\mu+\lambda-1}\Gamma(z)\Gamma\left(\frac{1}{2}+\lambda+\mu\right)}{\Gamma\left(\frac{1}{2}+\lambda-z\right)}-\frac{\sqrt{\pi}2^{\mu+z+\lambda}\Gamma(\lambda+\mu)\Gamma\left(\frac{1}{2}+z\right)}{x\Gamma(\lambda-z)}\nonumber\\
&=\frac{2^{\mu+\lambda-z}}{\sqrt{\pi}}\left(\frac{x}{\pi}\right)^{2z}\frac{\Gamma\left(z+\frac{1}{2}\right)\Gamma\left(\frac{1}{2}+\lambda+\mu+z\right)}{\Gamma\left(\frac{1}{2}+\lambda\right)}\sum_{n=1}^\infty\frac{1}{n^{2z+1}}\pFq{2}{1}{\frac{1}{2}+z,\frac{1}{2}+\lambda+\mu+z}{\frac{1}{2}+\lambda}{-\frac{x^2}{4\pi^2n^2}},
\end{align}
where ${}_2F_1(a,b;c;\xi)$ is the Gauss hypergeometric function defined by \cite[p.~110, Equation (5.4)]{temme}
\begin{align*}
{}_2F_1(a,b;c;\xi)=\frac{\Gamma(c)}{\Gamma(b)\Gamma(c-b)}\int_0^1 t^{b-1}(1-t)^{c-b-1}(1-t\xi)^{-a}\ dt,
\end{align*}
for $\mathrm{Re}(c)>\mathrm{Re}(b)>0,\ |\arg(1-\xi)|<\pi$.
\end{theorem}

For Theorem \ref{withmuknu} also, we remove the restriction on $z$ in the following theorem.
\begin{theorem}\label{muknuwanalytic}
Let $\mathrm{Re}(x)>0$ and $\mathrm{Re}(\mu+\lambda)>0$. Let $M>0$ be an integer. If $\mathrm{Re}(z)>-M$, then 
\begin{align}\label{muknuwanalyticeqn}
&2\sum_{n=1}^\infty\left(\frac{nx}{2}\right)^{z-\lambda}{}_{\mu}K_{z}(nx,\lambda)+\frac{2^{z+\mu+\lambda-1}\Gamma(z)\Gamma\left(\frac{1}{2}+\lambda+\mu\right)}{\Gamma\left(\frac{1}{2}+\lambda-z\right)}-\frac{\sqrt{\pi}2^{\mu+z+\lambda}\Gamma(\lambda+\mu)\Gamma\left(z+\frac{1}{2}\right)}{x\Gamma(\lambda-z)}\nonumber\\
&=\frac{2^{\mu+\lambda-z}}{\sqrt{\pi}}\left(\frac{x}{\pi}\right)^{2z}\frac{\Gamma\left(z+\frac{1}{2}\right)\Gamma\left(\frac{1}{2}+\lambda+\mu+z\right)}{\Gamma\left(\frac{1}{2}+\lambda\right)}\sum_{n=1}^\infty\frac{1}{n^{2z+1}}\Bigg\{\pFq{2}{1}{\frac{1}{2}+z,\frac{1}{2}+\lambda+\mu+z}{\frac{1}{2}+\lambda}{-\frac{x^2}{4\pi^2n^2}}\nonumber\\
&\qquad-\sum_{m=0}^{M-1}\frac{\left(\frac{1}{2}+z\right)_m\left(\frac{1}{2}+\lambda+\mu+z\right)_m}{m!\left(\frac{1}{2}+\lambda\right)_m}\left({-\frac{x^2}{4\pi^2n^2}}\right)^{m}\Bigg\}\nonumber\\
&\qquad+\frac{2^{\mu+\lambda-z}}{\sqrt{\pi}}\left(\frac{x}{\pi}\right)^{2z}\sum_{m=0}^{M-1}\frac{\Gamma\left(z+\frac{1}{2}+m\right)\Gamma\left(\frac{1}{2}+\lambda+\mu+z+m\right)\zeta(2z+2m+1)}{m!\Gamma\left(\frac{1}{2}+\lambda+m\right)}\left({-\frac{x^2}{4\pi^2}}\right)^{m}.
\end{align}
\end{theorem}

We separately record $z=0$ case of the above theorem below.
\begin{theorem}\label{nu=0}
Let $\mathrm{Re}(x)>0$ and $\mathrm{Re}(\mu+\lambda)>0$. We have
\begin{align}\label{nu=0eqn}
&2\sum_{n=1}^\infty\left(\frac{nx}{2}\right)^{-\lambda}{}_{\mu}K_{0}(nx,\lambda)-\frac{\pi 2^{\mu+\lambda}\Gamma(\lambda+\mu)}{x\Gamma(\lambda)}\nonumber\\
&=2^{\mu+\lambda}\frac{\Gamma\left(\frac{1}{2}+\lambda+\mu\right)}{\Gamma\left(\frac{1}{2}+\lambda\right)}\sum_{n=1}^\infty\frac{1}{n}\Bigg\{\pFq{2}{1}{\frac{1}{2},\frac{1}{2}+\lambda+\mu}{\frac{1}{2}+\lambda}{-\frac{x^2}{4\pi^2n^2}}-1\Bigg\}\nonumber\\
&\qquad+\frac{2^{\mu+\lambda-1}\Gamma\left(\frac{1}{2}+\lambda+\mu\right)}{\Gamma\left(\lambda+\frac{1}{2}\right)}\left(2\left(\gamma-\log\left(\frac{x}{4\pi}\right)\right)-\psi\left(\lambda+\frac{1}{2}\right)+\psi\left(\lambda+\mu+\frac{1}{2}\right)\right).
\end{align}
\end{theorem}

Theorem \ref{nu=0} gives a formula of Watson \cite[p.~301, Equation (6)]{watsonself}.
\begin{corollary}\label{nu=0ofmuknu}
For $\mathrm{Re}(x)>0$,
\begin{align*}
2\sum_{n=1}^\infty K_{0}(nx)-\frac{\pi}{x}=\sum_{n=1}^\infty\left\{\frac{2\pi}{\sqrt{x^2+4n^2\pi^2}}-\frac{1}{n}\right\}+\gamma-\log\left(\frac{x}{4\pi}\right).
\end{align*}
\end{corollary}

This paper is organized as follows. Theorem \ref{withkzw} and Theorem \ref{continuationwithkzw} are proved in Section \ref{swithkzw}. Section \ref{newproof} is devoted to giving a new proof of Theorem \ref{dkm}. We derive Theorem \ref{withmuknu}, and Theorem \ref{muknuwanalytic} and its Corollary in Section \ref{swithmuknu} and its subsection \ref{ssmuknu}.

\section{Generalization of Watson's result with $K_{z,w}(x)$}\label{swithkzw}

The main ingredient to prove Theorem \ref{withkzw} and Theorem \ref{withmuknu} is Poisson's summation formula in the following form \cite[p.~60-61]{titchfourier}:
\begin{theorem}\label{poisson}
If $f(t)$ is continuous and of bounded variation on $(0,\infty)$, and if $\int_0^\infty f(t)\ dt$ exists, then
\begin{align}
f(0^+)+2\sum_{n=1}^\infty f(n)=2\int_0^\infty f(t)\ dt+4\sum_{n=1}^\infty\int_0^\infty f(t)\cos(2\pi nt)\ dt.\nonumber
\end{align}

\end{theorem}

One of the nice properties of $K_{z,w}(t)$, derived in \cite{dkmm}, is the following integral representation for it \cite[Theorem 1.7]{dkmm}.
Let $z,w\in\mathbb{C}$ and $|\arg(x)|<\frac{\pi}{4}$, we have
\begin{align}\label{intkzw}
K_{z,w}(x) &=\left(\frac{x}{2}\right)^{-z}\int_0^\infty e^{-u^2-\frac{x^2}{u^2}}\cos(wu)\cos\left(\frac{wx}{u}\right)u^{2z-1}\ du.
\end{align}

To prove Theorem \ref{withkzw}, it is imperative to obtain the following new integral evaluation which contain $K_{z,w}(x)$ in its integrand. 
\begin{lemma}\label{kzwintevaluation}
Let $w\in\mathbb{C}$, $\mathrm{Re}(z)>0$, and $|\arg(x)|<\frac{\pi}{4}$. Let $\mathcal{A}\left(n,z,w,x\right)$ be defined in \eqref{Ax}. If $a\geq0$, then
\begin{align}\label{kzwintevaluationeqn}
\int_0^\infty \left(\frac{xt}{2}\right)^zK_{z,w}(xt) \cos(at)\ dt&=\frac{1}{4}\sqrt{\pi}\Gamma\left(\frac{1}{2}+z\right)\frac{x^{2z}e^{-\frac{w^2}{4}}}{(x^2+a^2)^{z+\frac{1}{2}}}\mathcal{A}\left(\frac{a}{2\pi},z,w,x\right).
\end{align}
\end{lemma}
\begin{proof}
Invoke \eqref{intkzw} to deduce that
\begin{align}\label{12q}
\int_0^\infty \left(\frac{xt}{2}\right)^zK_{z,w}(xt) \cos(at)\ dt&=\int_0^\infty\int_0^\infty e^{-u^2-\frac{x^2t^2}{4u^2}}\cos(wu)\cos\left(\frac{wxt}{2u}\right)u^{2z-1}\cos(at)\ dudt\nonumber\\
&=\int_0^\infty e^{-u^2}\cos(wu)u^{2z-1}\int_0^\infty e^{-\frac{x^2t^2}{4u^2}}\cos\left(\frac{wxt}{2u}\right)\cos(at)\ dtdu,
\end{align}
where in the last step we interchanged the order of the integration which is justified because of the absolute convergence. Note that 
\begin{align}\label{notethat}
\int_0^\infty e^{-\frac{x^2t^2}{4u^2}}\cos\left(\frac{wxt}{2u}\right)\cos(at)\ dt&=\frac{1}{2}\int_0^\infty e^{-\frac{x^2t^2}{4u^2}}\cos\left(at+\frac{wxt}{2u}\right)\ dt\nonumber\\
&\qquad+\frac{1}{2}\int_0^\infty e^{-\frac{x^2t^2}{4u^2}}\cos\left(at-\frac{wxt}{2u}\right)\ dt
\end{align}
From \cite[p.~488, Equation (3.896.4)]{grn}, for Re$(\beta)>0$, we have
\begin{align}\label{coscos}
\int_0^\infty e^{-\beta t^2}\cos\left(b t\right)\ dt=\frac{1}{2}\sqrt{\frac{\pi}{\beta}}\exp\left(-\frac{b^2}{4\beta}\right)
\end{align}
Let $\beta=\frac{x^2}{4u^2}$ and $b=at+\frac{wx}{2u}$ in \eqref{coscos} to get
\begin{align}\label{coscos1}
\int_0^\infty e^{-\frac{x^2t^2}{4u^2}}\cos\left(at+\frac{wxt}{2u}\right)\ dt=\frac{\sqrt{\pi}}{x}u\exp\left({-\frac{(2au+wx)^2}{4x^2}}\right).
\end{align}
Again use  \eqref{coscos} with $\beta=\frac{x^2}{4u^2}$ and $b=at+=-\frac{wx}{2u}$ to find
\begin{align}\label{coscos2}
\int_0^\infty e^{-\frac{x^2t^2}{4u^2}}\cos\left(at-\frac{wxt}{2u}\right)\ dt=\frac{\sqrt{\pi}}{x}u\exp\left({-\frac{(2au-wx)^2}{4x^2}}\right).
\end{align}
From \eqref{notethat}, \eqref{coscos1}, and \eqref{coscos2}, 
\begin{align}\label{1q}
\int_0^\infty e^{-\frac{x^2t^2}{4u^2}}\cos\left(\frac{wxt}{2u}\right)\cos(at)\ dt&=\frac{\sqrt{\pi}}{2x}u\left(e^{-\frac{(2au-wx)^2}{4x^2}}+e^{-\frac{(2au+wx)^2}{4x^2}}\right)\nonumber\\
&=\frac{\sqrt{\pi}}{2x}ue^{-\frac{w^2}{4}}e^{-\frac{a^2u^2}{x^2}}\left(e^{\frac{auw}{x}}+e^{-\frac{auw}{x}}\right)\nonumber\\
&=\frac{\sqrt{\pi}}{x}ue^{-\frac{w^2}{4}}e^{-\frac{a^2u^2}{x^2}}\cosh\left(\frac{auw}{x}\right).
\end{align}
Substitute value from \eqref{1q} in \eqref{12q} so as to obtain
\begin{align}\label{1q3}
\int_0^\infty \left(\frac{xt}{2}\right)^zK_{z,w}(xt) \cos(at)\ dt&=\frac{\sqrt{\pi}}{x}e^{-\frac{w^2}{4}}\int_0^\infty e^{-\left(\frac{a^2}{x^2}+1\right)u^2}\cos(wu)\cosh\left(\frac{auw}{x}\right)u^{2z}\ du\nonumber\\
&=\frac{\sqrt{\pi}}{2x}e^{-\frac{w^2}{4}}\Bigg\{\int_0^\infty e^{-\left(\frac{a^2}{x^2}+1\right)u^2}u^{2z}\cos\left(wu+\frac{iauw}{x}\right)\ du\nonumber\\
&\qquad+\int_0^\infty e^{-\left(\frac{a^2}{x^2}+1\right)u^2}u^{2z}\cos\left(wu-\frac{iauw}{x}\right)\ du\Bigg\},
\end{align}
where in the last step we used the elementary fact $2\cos(A)\cos(B)=\cos(A+B)+\cos(A-B)$. From \cite[p.~503, Equation (3.952.8)]{grn}\footnote{Condition on $\nu$ is given $\nu>0$, however, it is easy to see that this result is actually true for all $\nu\in\mathbb{C}$.}, for Re$(\beta)>0$, Re$(\mu)>0$ and $\nu\in\mathbb{C}$,  we have
\begin{align}\label{grad}
\int_0^\infty u^{\mu-1}e^{-\beta u^2}\cos(\nu u)\ du=\frac{1}{2}\beta^{-\frac{\mu}{2}}\Gamma\left(\frac{\mu}{2}\right)e^{-\frac{\nu^2}{4\beta}}{}_1F_1\left(\frac{1-\mu}{2};\frac{1}{2};\frac{\nu^2}{4\beta}\right).
\end{align}
Let $\mu=2z+1,\ \beta=\frac{a^2}{x^2}+1$ and $\nu=wu+\frac{iauw}{x}$ in \eqref{grad} to get
\begin{align}\label{gradev1}
&\int_0^\infty e^{-\left(\frac{a^2}{x^2}+1\right)u^2}u^{2z}\cos\left(wu+\frac{iauw}{x}\right)\ du\nonumber\\
&=\frac{1}{2}\Gamma\left(\frac{1}{2}+z\right)\frac{x^{2z}}{(x^2+a^2)^{z+\frac{1}{2}}}{}_1F_1\left(z+\frac{1}{2};\frac{1}{2};\frac{w^2(a-ix)}{4(a+ix)}\right).
\end{align}
Now let $\mu=2z+1,\ \beta=\frac{a^2}{x^2}+1$ and $\nu=wu-\frac{iauw}{x}$ in \eqref{grad} so that
\begin{align}\label{gradev2}
&\int_0^\infty e^{-\left(\frac{a^2}{x^2}+1\right)u^2}u^{2z}\cos\left(wu-\frac{iauw}{x}\right)\ du\nonumber\\
&=\frac{1}{2}\Gamma\left(\frac{1}{2}+z\right)\frac{x^{2z}}{(x^2+a^2)^{z+\frac{1}{2}}}{}_1F_1\left(z+\frac{1}{2};\frac{1}{2};\frac{w^2(a+ix)}{4(a-ix)}\right).
\end{align}
Finally substitute values from \eqref{gradev1} and \eqref{gradev2} in \eqref{1q3} and use \eqref{Ax} to arrive at \eqref{kzwintevaluationeqn}.
\end{proof}

As a special case of Lemma \ref{kzwintevaluation}, we get \cite[p.~299]{watsonself}.
\begin{corollary}\label{cor}
Let $z\in\mathbb{C}$ and $|\arg(x)|<\frac{\pi}{4}$. If $a\geq0$, then
\begin{align}\label{coreqn}
\int_0^\infty \left(\frac{xt}{2}\right)^zK_{z}(xt) \cos(at)\ dt&=\frac{1}{2}\sqrt{\pi}\Gamma\left(\frac{1}{2}+z\right)\frac{x^{2z}}{(x^2+a^2)^{z+\frac{1}{2}}}.
\end{align}
\end{corollary}
\begin{proof}
Let $w=0$ in Lemma \ref{kzwintevaluation} and use the fact that $K_{z,0}(xt)=K_{z}(xt)$ and ${}_1F_1(b;c;0)=1$ to get \eqref{coreqn}.
\end{proof}

We now have all necessary results to prove Theorem \ref{withkzw}.

\begin{proof}[Theorem \textup{\ref{withkzw}}][]
Let $f(t)=\left(\frac{xt}{2}\right)^zK_{z,w}(xt)$ in Theorem \ref{poisson}.  From \cite[Theorem 1.13(i)]{dkmm}, as $x\to0$,
\begin{align}\label{smallx}
K_{z,w}(x)\sim\frac{1}{2}\Gamma(z)\left(\frac{x}{2}\right)^{-z}{}_1F_1\left(z;\frac{1}{2};-\frac{w^2}{4}\right).
\end{align}
From \cite[Theorem 1.12]{dkmm}, we known that $\left(\frac{xt}{2}\right)^zK_{z,w}(xt)$ has exponential decay, therefore, along with \eqref{smallx} it easy to see that the integral $\int_0^\infty f(t)\ dt$ converges. Now by using \eqref{smallx}, we see that 
\begin{align}\label{f0plus}
\lim_{x\to0}\left(\left(\frac{xt}{2}\right)^zK_{z,w}(xt)\right)=f(0^+)=\frac{1}{2}\Gamma(z){}_1F_1\left(z;\frac{1}{2};-\frac{w^2}{4}\right).
\end{align}
Let $a=0$ in Lemma \ref{kzwintevaluation} to find
\begin{align}\label{witha=0}
\int_0^\infty f(t)\ dt&=\frac{1}{2x}\sqrt{\pi}\Gamma\left(\frac{1}{2}+z\right)e^{-\frac{w^2}{4}}{}_1F_1\left(z+\frac{1}{2};\frac{1}{2};-\frac{w^2}{4}\right).
\end{align}
Again invoke Theorem \ref{kzwintevaluationeqn} with $a=2\pi n$ so that
{\allowdisplaybreaks\begin{align}\label{withcos}
\int_0^\infty f(t) \cos(2\pi nt)\ dt&=\frac{1}{4}\sqrt{\pi}\Gamma\left(\frac{1}{2}+z\right)\frac{x^{2z}e^{-\frac{w^2}{4}}}{(x^2+4\pi^2n^2)^{z+\frac{1}{2}}}\mathcal{A}\left(n,z,w,x\right).
\end{align}}
Substitute values from \eqref{f0plus}, \eqref{witha=0} and \eqref{withcos} in Theorem \ref{poisson} to establish \eqref{withkzweqn}.
\end{proof}

\subsection{Analytic continuation of Theorem \ref{withkzw}}\label{sswithkzw}

This subsection is dedicated to finding analytic continuation of Theorem \ref{withkzw}, that is, to prove Theorem \ref{continuationwithkzw}.

\begin{proof}[Theorem \textup{\ref{continuationwithkzw}}][]
Note that, for large $n$,
\begin{align}\label{boundonrational}
\frac{1}{(x^2+4n^2\pi^2)^{z+\frac{1}{2}}}&=\frac{1}{(2n\pi)^{2z+1}}\left(1+\frac{x^2}{4n^2\pi^2}\right)^{-z-\frac{1}{2}}\nonumber\\
&=\frac{1}{(2n\pi)^{2z+1}}\sum_{m=0}^\infty\binom{-z-\frac{1}{2}}{m}\left(\frac{x^2}{4n^2\pi^2}\right)^m\nonumber\\
&=\frac{1}{(2n\pi)^{2z+1}}\sum_{m=0}^{M-1}\binom{-z-\frac{1}{2}}{m}\left(\frac{x^2}{4n^2\pi^2}\right)^m+O\left(\frac{1}{n^{2z+2M+1}}\right).
\end{align}
Now observe that
{\allowdisplaybreaks\begin{align}\label{abound}
\sum_{n=1}^\infty\frac{\mathcal{A}(n,z,w,x)}{(x^2+4n^2\pi^2)^{z+\frac{1}{2}}}&=\sum_{n=1}^\infty \mathcal{A}(n,z,w,x)\Bigg\{\frac{1}{(x^2+4n^2\pi^2)^{z+\frac{1}{2}}}-\frac{1}{(2n\pi)^{2z+1}}\sum_{m=0}^{M-1}\binom{-z-\frac{1}{2}}{m}\nonumber\\
&\qquad\times\left(\frac{x^2}{4n^2\pi^2}\right)^m+\frac{1}{(2n\pi)^{2z+1}}\sum_{m=0}^{M-1}\binom{-z-\frac{1}{2}}{m}\left(\frac{x^2}{4n^2\pi^2}\right)^m\Bigg\}\nonumber\\
&=\sum_{n=1}^\infty \mathcal{A}(n,z,w,x)\Bigg\{\frac{1}{(x^2+4n^2\pi^2)^{z+\frac{1}{2}}}-\sum_{m=0}^{M-1}\binom{-z-\frac{1}{2}}{m}\frac{x^{2m}}{(2n\pi)^{2z+2m+1}}\Bigg\}\nonumber\\
&\qquad+\sum_{n=1}^\infty \mathcal{A}(n,z,w,x)\sum_{m=0}^{M-1}\binom{-z-\frac{1}{2}}{m}\frac{x^{2m}}{(2n\pi)^{2z+2m+1}}\nonumber\\
&=\sum_{n=1}^\infty \mathcal{A}(n,z,w,x)\Bigg\{\frac{1}{(x^2+4n^2\pi^2)^{z+\frac{1}{2}}}-\sum_{m=0}^{M-1}\binom{-z-\frac{1}{2}}{m}\frac{x^{2m}}{(2n\pi)^{2z+2m+1}}\Bigg\}\nonumber\\
&\qquad+\sum_{m=0}^{M-1}\binom{-z-\frac{1}{2}}{m}\frac{x^{2m}}{(2\pi)^{2z+2m+1}}\sum_{n=1}^\infty \frac{\mathcal{A}(n,z,w,x)}{n^{2z+2m+1}}.
\end{align}}
Substitute \eqref{abound} in \eqref{withkzw} to obtain
\begin{align}\label{done}
&2\sum_{n=1}^\infty\left(\frac{nx}{2}\right)^zK_{z,w}(nx)+\frac{1}{2}\Gamma(z){}_1F_1\left(z;\frac{1}{2};-\frac{w^2}{4}\right)-\frac{\sqrt{\pi}}{x}\Gamma\left(\frac{1}{2}+z\right)e^{-\frac{w^2}{4}}{}_1F_1\left(\frac{1}{2}+z;\frac{1}{2};-\frac{w^2}{4}\right)\nonumber\\
&=\sqrt{\pi}x^{2z}\Gamma\left(\frac{1}{2}+z\right)e^{-\frac{w^2}{4}}\Bigg[\sum_{n=1}^\infty \mathcal{A}(n,z,w,x)\Bigg\{\frac{1}{(x^2+4n^2\pi^2)^{z+\frac{1}{2}}}-\sum_{m=0}^{M-1}\binom{-z-\frac{1}{2}}{m}\frac{x^{2m}}{(2n\pi)^{2z+2m+1}}\Bigg\}\nonumber\\
&\qquad+\sum_{m=0}^{M-1}\binom{-z-\frac{1}{2}}{m}\frac{x^{2m}}{(2\pi)^{2z+2m+1}}\sum_{n=1}^\infty \frac{\mathcal{A}(n,z,w,x)}{n^{2z+2m+1}}\Bigg].
\end{align}
By employing the asymptotic expansion of $K_{z,w}(x)$, given in \cite[Theorem 1.12]{dkmm}, it is easy to see that as the series on the right-hand side of \eqref{done} converges uniformly as a function of $z$ and its summand is also analytic, therefore, by using Weierstrass' theorem on analytic function it represents an analytic function in Re$(z)>-M$. Hence the right-hand side of \eqref{done} is analytic in $z$ in Re$(z)>-M$.

By using \eqref{boundonrational} and the fact $\mathcal{A}(n,z,w,x)=O(1)$, for large $n$, we get 
\begin{align}\label{bigoh}
\mathcal{A}(n,z,w,x)\Bigg\{\frac{1}{(x^2+4n^2\pi^2)^{z+\frac{1}{2}}}-\sum_{m=0}^{M-1}\binom{-z-\frac{1}{2}}{m}\frac{x^{2m}}{(2n\pi)^{2z+2m+1}}\Bigg\}=O\left(\frac{1}{n^{2z+2M+1}}\right).
\end{align}
Upon employing \eqref{bigoh} is it is easy to see that the infinite series on the left-hand side of \eqref{done} is uniformly convergent as a function of $z$ in Re$(z)>-M$. The summand of this series is also analytic in this region. Therefore, the left-hand side of \eqref{done} represents an analytic function of $z$ in Re$(z)>-M$. Now results follows by the principle of analytic continuation for the conditions imposed in the hypotheses of the theorem.
\end{proof}

\begin{proof}[Corollary \textup{\ref{spwatson}}][]
Let $w=0$ in Theorem \ref{continuationwithkzw}. Use the fact $K_{z,0}(x)=K_{z}(x)$ and $ \mathcal{A}(n,z,0,x)=2$ to arrive at \eqref{spwatsoneqn}.
\end{proof}

\section{A new proof of Dixit-Kesarwani-Moll's generalization of the Ramanujan-Guinand formula}\label{newproof}
 
 In this section, as an application of our Theorem \ref{withkzw}, we provide a new proof of Dixit-Kesarwani-Moll's generalization of the Ramanujan-Guinand formula, that is, Theorem \ref{dkm}.

\begin{proof}[Theorem \textup{\ref{dkm}}][]
Using the definition of the divisor function $\sigma_{-z}(n)=\sum_{d|n}d^{-z}$, note that
\begin{align}\label{beforewatson}
\sqrt{\alpha}\sum_{n=1}^\infty \sigma_{-z}(n)n^{\frac{z}{2}}e^{-\frac{w^2}{4}}K_{\frac{z}{2},iw}(2n\alpha)&=\sqrt{\alpha}e^{-\frac{w^2}{4}}\sum_{n=1}^\infty \sum_{d|n}d^{-z}n^{\frac{z}{2}}K_{\frac{z}{2},iw}(2n\alpha)\nonumber\\
&=\sqrt{\alpha}e^{-\frac{w^2}{4}}\sum_{d=1}^\infty d^{-\frac{z}{2}}\sum_{k=1}^\infty k^{\frac{z}{2}}K_{\frac{z}{2},iw}(2kd\alpha).
\end{align}
Let $x=2d\alpha$ and replace $w$ by $iw$ in Theorem \ref{withkzw} to see 
{\allowdisplaybreaks\begin{align}\label{lemmaequ}
\sum_{n=1}^\infty n^{\frac{z}{2}}K_{\frac{z}{2},iw}(2nd\alpha)&=-\frac{1}{4}\frac{\alpha^{-\frac{z}{2}}}{d^{\frac{z}{2}}}\Gamma\left(\frac{z}{2}\right){}_1F_1\left(\frac{z}{2};\frac{1}{2};\frac{w^2}{4}\right)+\frac{1}{4}\sqrt{\pi}e^{\frac{w^2}{4}}\frac{\alpha^{-\frac{z}{2}-1}}{d^{\frac{z}{2}+1}}\Gamma\left(\frac{z+1}{2}\right)\nonumber\\
&\times{}_1F_1\left(\frac{1+z}{2};\frac{1}{2};\frac{w^2}{4}\right)+\frac{\sqrt{\pi}}{2^{1-z}}\frac{\alpha^{\frac{z}{2}}}{d^{-\frac{z}{2}}}\Gamma\left(\frac{1+z}{2}\right)e^{\frac{w^2}{4}}\sum_{n=1}^\infty\frac{\mathcal{A}\left(n,\frac{z}{2},iw,2d\alpha\right)}{(4d^2\alpha^2+4\pi^2n^2)^{\frac{z+1}{2}}}.
\end{align}}
Substitute value from \eqref{lemmaequ} in \eqref{beforewatson} to conclude 
\begin{align}
\sqrt{\alpha}\sum_{n=1}^\infty \sigma_{-z}(n)n^{\frac{z}{2}}e^{-\frac{w^2}{4}}K_{\frac{z}{2},iw}(2n\alpha)&=\sqrt{\alpha}e^{-\frac{w^2}{4}}\Bigg\{-\frac{1}{4}\alpha^{-\frac{z}{2}}\Gamma\left(\frac{z}{2}\right){}_1F_1\left(\frac{z}{2};\frac{1}{2};\frac{w^2}{4}\right)\sum_{d=1}^\infty\frac{1}{d^{z}}\nonumber\\
&\quad+\frac{1}{4}\sqrt{\pi}e^{\frac{w^2}{4}}\alpha^{-\frac{z}{2}-1}\Gamma\left(\frac{z+1}{2}\right){}_1F_1\left(\frac{1+z}{2};\frac{1}{2};\frac{w^2}{4}\right)\sum_{d=1}^\infty\frac{1}{d^{z+1}}\nonumber\\
&\quad+\frac{1}{4}\sqrt{\pi}e^{\frac{w^2}{4}}\alpha^{\frac{z}{2}}\Gamma\left(\frac{1+z}{2}\right)\sum_{d=1}^\infty\sum_{n=1}^\infty\frac{\mathcal{A}\left(n,\frac{z}{2},iw,2d\alpha\right)}{(d^2\alpha^2+\pi^2n^2)^{\frac{z+1}{2}}}\Bigg\}.\nonumber
\end{align}
Use the definition of the Riemann zeta function $\zeta(z)$ in the above equation to get
\begin{align}\label{alphaform}
\sqrt{\alpha}\sum_{n=1}^\infty \sigma_{-z}(n)n^{\frac{z}{2}}e^{-\frac{w^2}{4}}K_{\frac{z}{2},iw}(2n\alpha)&=-\frac{1}{4}\alpha^{\frac{1-z}{2}}\Gamma\left(\frac{z}{2}\right)\zeta(z)e^{-\frac{w^2}{4}}{}_1F_1\left(\frac{z}{2};\frac{1}{2};\frac{w^2}{4}\right)\nonumber\\
&\quad+\frac{1}{4}\sqrt{\pi}\alpha^{\frac{-z-1}{2}}\Gamma\left(\frac{z+1}{2}\right)\zeta(z+1){}_1F_1\left(\frac{1+z}{2};\frac{1}{2};\frac{w^2}{4}\right)\nonumber\\
&\quad+\frac{1}{4}\sqrt{\pi}\alpha^{\frac{z+1}{2}}\Gamma\left(\frac{1+z}{2}\right)\sum_{d=1}^\infty\sum_{n=1}^\infty\frac{\mathcal{A}\left(n,\frac{z}{2},iw,2d\alpha\right)}{(d^2\alpha^2+\pi^2n^2)^{\frac{z+1}{2}}}.
\end{align}
Upon using the fact $\alpha\beta=\pi^2$ in the last term of \eqref{alphaform}, we find that
\begin{align}\label{1alpha}
\sqrt{\alpha}\sum_{n=1}^\infty \sigma_{-z}(n)n^{\frac{z}{2}}e^{-\frac{w^2}{4}}K_{\frac{z}{2},iw}(2n\alpha)&=-\frac{1}{4}\alpha^{\frac{1-z}{2}}\Gamma\left(\frac{z}{2}\right)\zeta(z)e^{-\frac{w^2}{4}}{}_1F_1\left(\frac{z}{2};\frac{1}{2};\frac{w^2}{4}\right)\nonumber\\
&\quad+\frac{1}{4}\sqrt{\pi}\alpha^{\frac{-z-1}{2}}\Gamma\left(\frac{z+1}{2}\right)\zeta(z+1){}_1F_1\left(\frac{1+z}{2};\frac{1}{2};\frac{w^2}{4}\right)\nonumber\\
&\quad+\frac{1}{4}\sqrt{\pi}\beta^{\frac{z+1}{2}}\Gamma\left(\frac{1+z}{2}\right)\sum_{d=1}^\infty\sum_{n=1}^\infty\frac{\mathcal{A}\left(n,\frac{z}{2},iw,2d\alpha\right)}{(d^2\pi^2+n^2\beta^2)^{\frac{z+1}{2}}}.
\end{align}
Replace $\alpha$ by $\beta$ and $w$ by $iw$ in \eqref{alphaform} so that
{\allowdisplaybreaks\begin{align}\label{betaform}
\sqrt{\beta}\sum_{n=1}^\infty \sigma_{-z}(n)n^{\frac{z}{2}}e^{\frac{w^2}{4}}K_{\frac{z}{2},w}(2n\beta)&=-\frac{1}{4}\beta^{\frac{1-z}{2}}\Gamma\left(\frac{z}{2}\right)\zeta(z)e^{\frac{w^2}{4}}{}_1F_1\left(\frac{z}{2};\frac{1}{2};-\frac{w^2}{4}\right)\nonumber\\
&\quad+\frac{1}{4}\sqrt{\pi}\beta^{\frac{-z-1}{2}}\Gamma\left(\frac{z+1}{2}\right)\zeta(z+1){}_1F_1\left(\frac{1+z}{2};\frac{1}{2};-\frac{w^2}{4}\right)\nonumber\\
&\quad+\frac{1}{4}\sqrt{\pi}\beta^{\frac{z+1}{2}}\Gamma\left(\frac{1+z}{2}\right)\sum_{d=1}^\infty\sum_{n=1}^\infty\frac{\mathcal{A}\left(n,\frac{z}{2},w,2d\beta\right)}{(n^2\pi^2+d^2\beta^2)^{\frac{z+1}{2}}},
\end{align}}
where we used the facts that $K_{\frac{z}{2},w}(2n\beta)$ and $\mathcal{A}\left(n,\frac{z}{2},w,2d\beta\right)$ are even functions of $w$. Now interchange the role of the variables $n$ and $d$ in the last term of \eqref{betaform} and then interchange the order of summation to get
{\allowdisplaybreaks\begin{align}\label{1betaform}
\sqrt{\beta}\sum_{n=1}^\infty \sigma_{-z}(n)n^{\frac{z}{2}}e^{\frac{w^2}{4}}K_{\frac{z}{2},w}(2n\beta)&=-\frac{1}{4}\beta^{\frac{1-z}{2}}\Gamma\left(\frac{z}{2}\right)\zeta(z)e^{\frac{w^2}{4}}{}_1F_1\left(\frac{z}{2};\frac{1}{2};-\frac{w^2}{4}\right)\nonumber\\
&\quad+\frac{1}{4}\sqrt{\pi}\beta^{\frac{-z-1}{2}}\Gamma\left(\frac{z+1}{2}\right)\zeta(z+1){}_1F_1\left(\frac{1+z}{2};\frac{1}{2};-\frac{w^2}{4}\right)\nonumber\\
&\quad+\frac{1}{4}\sqrt{\pi}\beta^{\frac{z+1}{2}}\Gamma\left(\frac{1+z}{2}\right)\sum_{d=1}^\infty\sum_{n=1}^\infty\frac{\mathcal{A}\left(d,\frac{z}{2},w,2n\beta\right)}{(d^2\pi^2+n^2\beta^2)^{\frac{z+1}{2}}}.
\end{align}}
From the definition \eqref{Ax} of $\mathcal{A}\left(n,\frac{z}{2},iw,2d\alpha\right)$ and the fact $\alpha\beta=\pi^2$, we have
{\allowdisplaybreaks\begin{align}\label{Asymm}
\mathcal{A}\left(n,\frac{z}{2},iw,2d\alpha\right)&={}_1F_1\left(\frac{1}{2}+z;\frac{1}{2};-\frac{w^2(2n\pi+2id\alpha)}{8n\pi-8id\alpha}\right)+{}_1F_1\left(\frac{1}{2}+z;\frac{1}{2};-\frac{w^2(2n\pi-2id\alpha)}{8n\pi+8id\alpha}\right)\nonumber\\
&={}_1F_1\left(\frac{1}{2}+z;\frac{1}{2};-\frac{w^2(2n\pi+2id\pi^2/\beta)}{8n\pi-8id\pi^2/\beta}\right)+{}_1F_1\left(\frac{1}{2}+z;\frac{1}{2};-\frac{w^2(2n\pi-2id\pi^2/\beta)}{8n\pi+8id\pi^2/\beta}\right)\nonumber\\
&={}_1F_1\left(\frac{1}{2}+z;\frac{1}{2};-\frac{w^2(2n\beta+2id\pi)}{8n\beta-8id\pi}\right)+{}_1F_1\left(\frac{1}{2}+z;\frac{1}{2};-\frac{w^2(2n\beta-2id\pi)}{8n\beta+8id\pi}\right)\nonumber\\
&={}_1F_1\left(\frac{1}{2}+z;\frac{1}{2};-\frac{w^2(2n\beta i-2d\pi)}{8n\beta i+8d\pi}\right)+{}_1F_1\left(\frac{1}{2}+z;\frac{1}{2};-\frac{w^2(2n\beta i+2d\pi)}{8n\beta i-8d\pi}\right)\nonumber\\
&={}_1F_1\left(\frac{1}{2}+z;\frac{1}{2};\frac{w^2(2d\pi-2n\beta i)}{8n\beta i+8d\pi}\right)+{}_1F_1\left(\frac{1}{2}+z;\frac{1}{2};\frac{w^2(2n\beta i+2d\pi)}{8d\pi-8n\beta i}\right)\nonumber\\
&=\mathcal{A}\left(d,\frac{z}{2},w,2n\beta\right).
\end{align}}
Employ \eqref{Asymm} in \eqref{1betaform} and then subtract the resulting expression from \eqref{1alpha} and simplify to arrive at
\begin{align}\label{almostthere}
&\sqrt{\alpha}\sum_{n=1}^\infty \sigma_{-z}(n)n^{\frac{z}{2}}e^{-\frac{w^2}{4}}K_{\frac{z}{2},iw}(2n\alpha)-\sqrt{\beta}\sum_{n=1}^\infty \sigma_{-z}(n)n^{\frac{z}{2}}e^{\frac{w^2}{4}}K_{\frac{z}{2},w}(2n\beta)\nonumber\\
&=-\frac{1}{4}\alpha^{\frac{1-z}{2}}\Gamma\left(\frac{z}{2}\right)\zeta(z)e^{-\frac{w^2}{4}}{}_1F_1\left(\frac{z}{2};\frac{1}{2};\frac{w^2}{4}\right)+\frac{1}{4}\sqrt{\pi}\alpha^{\frac{-z-1}{2}}\Gamma\left(\frac{z+1}{2}\right)\zeta(z+1){}_1F_1\left(\frac{1+z}{2};\frac{1}{2};\frac{w^2}{4}\right)\nonumber\\
&+\frac{1}{4}\beta^{\frac{1-z}{2}}\Gamma\left(\frac{z}{2}\right)\zeta(z)e^{\frac{w^2}{4}}{}_1F_1\left(\frac{z}{2};\frac{1}{2};-\frac{w^2}{4}\right)-\frac{1}{4}\sqrt{\pi}\beta^{\frac{-z-1}{2}}\Gamma\left(\frac{z+1}{2}\right)\zeta(z+1){}_1F_1\left(\frac{1+z}{2};\frac{1}{2};-\frac{w^2}{4}\right).
\end{align}
We have Kummer's transformation \cite[P.~173, Equation (7.5)]{temme}
\begin{align}\label{kummar}
{}_1F_1(b;c;z)=e^{z}{}_1F_1(c-b;c;-z).
\end{align}
Use \eqref{kummar} twice in \eqref{almostthere} to find
\begin{align}\label{almostthere1}
&\sqrt{\alpha}\sum_{n=1}^\infty \sigma_{-z}(n)n^{\frac{z}{2}}e^{-\frac{w^2}{4}}K_{\frac{z}{2},iw}(2n\alpha)-\sqrt{\beta}\sum_{n=1}^\infty \sigma_{-z}(n)n^{\frac{z}{2}}e^{\frac{w^2}{4}}K_{\frac{z}{2},w}(2n\beta)\nonumber\\
&=-\frac{1}{4}\alpha^{\frac{1-z}{2}}\Gamma\left(\frac{z}{2}\right)\zeta(z){}_1F_1\left(\frac{1-z}{2};\frac{1}{2};-\frac{w^2}{4}\right)+\frac{1}{4}\sqrt{\pi}\alpha^{\frac{-z-1}{2}}\Gamma\left(\frac{z+1}{2}\right)\zeta(z+1){}_1F_1\left(\frac{1+z}{2};\frac{1}{2};\frac{w^2}{4}\right)\nonumber\\
&\quad+\frac{1}{4}\beta^{\frac{1-z}{2}}\Gamma\left(\frac{z}{2}\right)\zeta(z){}_1F_1\left(\frac{z}{2};\frac{1}{2};\frac{w^2}{4}\right)-\frac{1}{4}\sqrt{\pi}\beta^{\frac{-z-1}{2}}\Gamma\left(\frac{z+1}{2}\right)\zeta(z+1){}_1F_1\left(\frac{1+z}{2};\frac{1}{2};-\frac{w^2}{4}\right).
\end{align}
Now invoke the functional equation of $\zeta(s)$ \cite[p.22, Equation (2.6.40]{titch}
\begin{align*}
\pi^{-\frac{s}{2}}\Gamma\left(\frac{s}{2}\right)\zeta(s)=\pi^{-\frac{1-s}{2}}\Gamma\left(\frac{1-s}{2}\right)\zeta(1-s),
\end{align*}
along with letting $s=z+1$ and use the hypothesis $\alpha\beta=\pi^2$ so that
\begin{align*}
&\frac{1}{4}\sqrt{\pi}\alpha^{\frac{-z-1}{2}}\Gamma\left(\frac{z+1}{2}\right)\zeta(z+1){}_1F_1\left(\frac{1+z}{2};\frac{1}{2};\frac{w^2}{4}\right)-\frac{1}{4}\sqrt{\pi}\beta^{\frac{-z-1}{2}}\Gamma\left(\frac{z+1}{2}\right)\zeta(z+1)\nonumber\\
&\qquad\times{}_1F_1\left(\frac{1+z}{2};\frac{1}{2};-\frac{w^2}{4}\right)\nonumber\\
&=\frac{1}{4}\beta^{\frac{z+1}{2}}\Gamma\left(-\frac{z}{2}\right)\zeta(-z){}_1F_1\left(\frac{1+z}{2};\frac{1}{2};\frac{w^2}{4}\right)-\frac{1}{4}\alpha^{\frac{z+1}{2}}\Gamma\left(-\frac{z}{2}\right)\zeta(-z){}_1F_1\left(\frac{1+z}{2};\frac{1}{2};-\frac{w^2}{4}\right).
\end{align*}
Substitute the above value in \eqref{almostthere1} so as to deduce that
\begin{align*}
&\sqrt{\alpha}\sum_{n=1}^\infty \sigma_{-z}(n)n^{\frac{z}{2}}e^{-\frac{w^2}{4}}K_{\frac{z}{2},iw}(2n\alpha)-\sqrt{\beta}\sum_{n=1}^\infty \sigma_{-z}(n)n^{\frac{z}{2}}e^{\frac{w^2}{4}}K_{\frac{z}{2},w}(2n\beta)\nonumber\\
&=-\frac{1}{4}\alpha^{\frac{1-z}{2}}\Gamma\left(\frac{z}{2}\right)\zeta(z){}_1F_1\left(\frac{1-z}{2};\frac{1}{2};-\frac{w^2}{4}\right)+\frac{1}{4}\beta^{\frac{z+1}{2}}\Gamma\left(-\frac{z}{2}\right)\zeta(-z){}_1F_1\left(\frac{1+z}{2};\frac{1}{2};\frac{w^2}{4}\right)\nonumber\\
&\quad+\frac{1}{4}\beta^{\frac{1-z}{2}}\Gamma\left(\frac{z}{2}\right)\zeta(z){}_1F_1\left(\frac{z}{2};\frac{1}{2};\frac{w^2}{4}\right)-\frac{1}{4}\alpha^{\frac{z+1}{2}}\Gamma\left(-\frac{z}{2}\right)\zeta(-z){}_1F_1\left(\frac{1+z}{2};\frac{1}{2};-\frac{w^2}{4}\right).
\end{align*}
Now simplify the above equation to prove \eqref{dkmeqn}.
\end{proof}

\section{Generalization of Watson's result in the setting of ${}_{\mu}K_{z}(x,\lambda)$}\label{swithmuknu}

This section is devoted to proving Theorem \ref{withmuknu}. However, to prove it we need to evaluate the following integral involving the generalized modified Bessel function ${}_{\mu}K_{z}(x, \lambda)$.

\begin{lemma}\label{intmuknuw}
Let $\mathrm{Re}(z)>0$ and $x, \mu, \lambda\in\mathbb{C}$ such that $\mathrm{Re}(\mu+\lambda)>0$. Then for $a\geq0$, we have
\begin{align}
&\int_0^\infty \left(\frac{tx}{2}\right)^{z-\lambda}{}_\mu K_z(tx,\lambda)\cos(at)\ dt\nonumber\\
&=\frac{\sqrt{\pi}2^{\mu+z+\lambda-1}x^{2z}}{a^{2z+1}}\frac{\Gamma\left(\frac{1}{2}+z\right)\Gamma\left(\frac{1}{2}+\mu+z+\lambda\right)}{\Gamma\left(\frac{1}{2}+\lambda\right)}\pFq{2}{1}{\frac{1}{2}+z,\frac{1}{2}+\lambda+\mu+z}{\frac{1}{2}+\lambda}{-\frac{x^2}{a^2}}.\nonumber 
\end{align}
\end{lemma}
\begin{proof}
From \eqref{muknuwdef} employ the definition  of  ${}_\mu K_z(x,\lambda)$,
\begin{align}\label{split}
&\int_0^\infty \left(\frac{tx}{2}\right)^{z-\lambda}{}_\mu K_z(tx,\lambda)\cos(at)\ dt\nonumber\\
&=\frac{\pi 2^{\mu+\lambda+z-1}\Gamma\left(\mu+\lambda+\frac{1}{2}\right)}{\Gamma\left(\frac{1}{2}+\lambda-z\right)\Gamma\left(1-z\right)}\int_0^\infty\cos(at){}_1F_2\left(\mu+\lambda+\frac{1}{2};\frac{1}{2}+\lambda-z;1-z;\frac{(tx)^2}{4}\right)\ dt \nonumber\\
&-\frac{\pi x^{2z}2^{\mu+\lambda-z-1}\Gamma\left(\mu+z+\lambda+\frac{1}{2}\right)}{\sin(\pi z)\Gamma(1+z)\Gamma\left(\lambda+\frac{1}{2}\right)}\int_0^\infty t^{2z}\cos(at){}_1F_2\left(\mu+z+\lambda+\frac{1}{2};\lambda+\frac{1}{2},1+z;\frac{(tx)^2}{4}\right)\ dt.
\end{align}
From \cite[p.~818, Equation (7.542.3)]{grn}\footnote{minus sign in the argument of the ${}_pF_q$ on the right-hand side must replaced by the plus sign}, for $y\geq0, p\leq q-1$ and Re$(\sigma)>|\mathrm{Re}(\nu)|$, we have
\begin{align}\label{gennu}
&\int_0^\infty t^{\sigma-1}{}_pF_q(a_1,...,a_p;b_1,...,b_q;-\eta t^2)Y_\nu(yt)\ dt\nonumber\\
&=-\frac{2^{\sigma-1}}{\pi}y^{-\sigma}\cos\left(\frac{\pi}{2}(\sigma-\nu)\right)\Gamma\left(\frac{\sigma+\nu}{2}\right)\Gamma\left(\frac{\sigma-\nu}{2}\right){}_{p+2}F_q\left(a_1,...,a_p,\frac{\sigma+\nu}{2},\frac{\sigma-\nu}{2};b_1,...,b_q;\frac{4\eta}{y^2}\right).
\end{align}
Let $p=1,\ q=2,\ \nu=\frac{1}{2}$ in \eqref{gennu} and the fact that $Y_{\frac{1}{2}}(x)=-\sqrt{\frac{\pi}{2x}}\cos(x)$ to find
\begin{align}\label{nuhalf}
&\int_0^\infty t^{\sigma-\frac{3}{2}}{}_1F_2(a_1;b_1,b_2;-\eta t^2)\cos(yt)\ dt\nonumber\\
&=\frac{2^{\sigma-\frac{3}{2}}}{\sqrt{\pi}}y^{-\sigma}\cos\left(\frac{\pi}{2}\left(\sigma-\frac{1}{2}\right)\right)\Gamma\left(\frac{\sigma+1/2}{2}\right)\Gamma\left(\frac{\sigma-1/2}{2}\right)\nonumber\\
&\qquad\times{}_{p+2}F_q\left(a_1,\frac{\sigma+1/2}{2},\frac{\sigma-1/2}{2};b_1,b_2;\frac{4\eta}{y^2}\right).
\end{align}
Let $\sigma=\frac{3}{2},\ y=a$ and $a_1=\mu+\lambda+\frac{1}{2},\ b_1=\lambda+\frac{1}{2}-z,\ b_2=1-z$ and $\eta=-\frac{x^2}{4}$ in \eqref{nuhalf}, then note that cosine term on the right-hand side vanishes; therefore, we get
\begin{align}\label{intiszero}
&\int_0^\infty {}_1F_2\left(\mu+\lambda+\frac{1}{2};\lambda+\frac{1}{2}-z,1-z; \frac{x^2t^2}{4}\right)\cos(at)\ dt=0
\end{align}
Now let $\sigma=\frac{3}{2}+2z,\ y=a$ and $a_1=\mu+z+\lambda+\frac{1}{2},\ b_1=\lambda+\frac{1}{2},\ b_2=1+z$ and $\eta=-\frac{x^2}{4}$ in \eqref{nuhalf} to see that
\begin{align}\label{intisnotzero}
&\int_0^\infty t^{2z}{}_1F_2\left(\mu+z+\lambda+\frac{1}{2};\lambda+\frac{1}{2},1+z; \frac{x^2t^2}{4}\right)\cos(at)\ dt\nonumber\\
&=-\frac{\Gamma(2z+1)\sin(\pi z)}{a^{2z+1}}{}_2F_1\left(\frac{1}{2}+z,\frac{1}{2}+\lambda+\mu+z;\frac{1}{2}+\lambda;-\frac{x^2}{a^2}\right).
\end{align}
Lemma \ref{intmuknuw} now follows from \eqref{split}, \eqref{intiszero} and \eqref{intisnotzero}.
\end{proof}

We are now ready to prove Theorem \ref{withmuknu}.
\begin{proof}[Theorem \textup{\ref{withmuknu}}][]
Let 
\begin{align}\label{fmuknuw}
f(t):=\left(\frac{tx}{2}\right)^{z-\lambda}{}_\mu K_z(tx,\lambda).
\end{align}
From \eqref{muknuwdef} and \eqref{fmuknuw}, 
\begin{align}\label{f(t)}
f(t)&=\frac{\pi 2^{\lambda+\mu+z-1}}{\sin(z\pi)}\bigg\{\frac{\Gamma(\mu+\lambda+\tfrac{1}{2})}{\Gamma(1-z)\Gamma(\lambda+\tfrac{1}{2}-z)}\pFq12{\mu+\lambda+\tfrac{1}{2}}{\lambda+\tfrac{1}{2}-z,1-z}{\frac{(xt)^2}{4}}\nonumber\\
&\quad\quad\quad\quad\quad\quad-\left(\frac{xt}{2}\right)^{2z}\frac{\Gamma(\mu+z+\lambda+\tfrac{1}{2})}{\Gamma(1+z)\Gamma(\lambda+\tfrac{1}{2})}\pFq12{\mu+z+\lambda+\tfrac{1}{2}}{\lambda+\tfrac{1}{2},1+z}{\frac{(xt)^2}{4}}\bigg\}.
\end{align}
It is easy to see from the above equation that as $t\to0$, for Re$(z)>0$,
\begin{align}\label{ttends0}
f(t)=O(1).
\end{align}
From \cite[Lemma 7.1]{dkk}, as $t\to\infty$,
\begin{align}\label{tinf}
f(t)=O\left(t^{-2\lambda-2\mu-1}\right).
\end{align}
Now by using \eqref{ttends0} and \eqref{tinf}, it is clear that the integral $\int_0^\infty f(t)\ dt$ exists for Re$(z)>0$ and Re$(\mu+\lambda)>0$. 

Let $t=0$ and use $\Gamma(z)\Gamma(1-z)=\pi/\sin(\pi z)$ in \eqref{f(t)} so that, for Re$(z)>0$,
\begin{align}\label{at0}
f(0)=\frac{2^{\lambda+\mu+z-1}\Gamma(\mu+\lambda+\tfrac{1}{2})\Gamma(z)}{\Gamma(\lambda+\tfrac{1}{2}-z)}.
\end{align}
Employ Lemma \ref{intmuknuw} with $a=0$ and \eqref{fmuknuw} to get
\begin{align}\label{f}
\int_0^\infty f(t)\ dt=\frac{\sqrt{\pi}2^{\mu+z+\lambda-1}\Gamma(\mu+\lambda)\Gamma\left(\frac{1}{2}+z\right)}{x\Gamma(\lambda-z)}.
\end{align}
Again upon invoking Lemma \ref{intmuknuw} with $a=2\pi n$, we have 
\begin{align}\label{fcos}
\int_0^\infty f(t)\cos(2\pi nt)\ dt&=\frac{\sqrt{\pi}2^{\mu+z+\lambda-1}x^{2z}}{(2\pi n)^{2z+1}}\Gamma\left(\frac{1}{2}+z\right)\frac{\Gamma\left(\frac{1}{2}+\mu+z+\lambda\right)}{\Gamma\left(\frac{1}{2}+\lambda\right)}\nonumber\\
&\qquad\times{}_2F_1\left(\frac{1}{2}+z,\frac{1}{2}+\mu+z+\lambda;\frac{1}{2}+\lambda;-\frac{x^2}{4\pi^2n^2}\right).
\end{align}
Substitute values from \eqref{fmuknuw}, \eqref{at0}, \eqref{f} and \eqref{fcos} in Theorem \ref{poisson} and simplify to arrive at \eqref{withmuknueqn}.
\end{proof}

\subsection{Analytic continuation of Theorem \ref{withmuknu}}\label{ssmuknu}

\begin{proof}[Theorem \textup{\ref{muknuwanalytic}}][]
Upon using the series definition of ${}_2F_1$, we see that as $n\to\infty$,
\begin{align}\label{2f1bound}
\pFq{2}{1}{\frac{1}{2}+z,\frac{1}{2}+\lambda+\mu+z}{\frac{1}{2}+\lambda}{-\frac{x^2}{4\pi^2n^2}}&=\sum_{m=0}^M\frac{\left(\frac{1}{2}+z\right)_m\left(\frac{1}{2}+\lambda+\mu+z\right)_m}{m!\left(\frac{1}{2}+\lambda\right)_m}\left({-\frac{x^2}{4\pi^2n^2}}\right)^{-m}\nonumber\\
&\qquad+O\left({n^{-2M}}\right).
\end{align}
Add and subtract $\sum_{m=0}^M\frac{\left(\frac{1}{2}+z\right)_m\left(\frac{1}{2}+\lambda+\mu+z\right)_m}{m!\left(\frac{1}{2}+\lambda\right)_m}\left({-\frac{x^2}{4\pi^2n^2}}\right)^{-m}$ in the first step below to see that
\begin{align}\label{plusminus}
&\sum_{n=1}^\infty\frac{1}{n^{2z+1}}\pFq{2}{1}{\frac{1}{2}+z,\frac{1}{2}+\lambda+\mu+z}{\frac{1}{2}+\lambda}{-\frac{x^2}{4\pi^2n^2}}\nonumber\\
&=\sum_{n=1}^\infty\frac{1}{n^{2z+1}}\Bigg\{\pFq{2}{1}{\frac{1}{2}+z,\frac{1}{2}+\lambda+\mu+z}{\frac{1}{2}+\lambda}{-\frac{x^2}{4\pi^2n^2}}-\sum_{m=0}^{M-1}\frac{\left(\frac{1}{2}+z\right)_m\left(\frac{1}{2}+\lambda+\mu+z\right)_m}{m!\left(\frac{1}{2}+\lambda\right)_m}\nonumber\\
&\qquad\times\left({-\frac{x^2}{4\pi^2n^2}}\right)^{m}\Bigg\}+\sum_{n=1}^\infty\frac{1}{n^{2z+1}}\sum_{m=0}^{M-1}\frac{\left(\frac{1}{2}+z\right)_m\left(\frac{1}{2}+\lambda+\mu+z\right)_m}{m!\left(\frac{1}{2}+\lambda\right)_m}\left({-\frac{x^2}{4\pi^2n^2}}\right)^{m}\nonumber\\
&=\sum_{n=1}^\infty\frac{1}{n^{2z+1}}\Bigg\{\pFq{2}{1}{\frac{1}{2}+z,\frac{1}{2}+\lambda+\mu+z}{\frac{1}{2}+\lambda}{-\frac{x^2}{4\pi^2n^2}}-\sum_{m=0}^{M-1}\frac{\left(\frac{1}{2}+z\right)_m\left(\frac{1}{2}+\lambda+\mu+z\right)_m}{m!\left(\frac{1}{2}+\lambda\right)_m}\nonumber\\
&\qquad\times\left({-\frac{x^2}{4\pi^2n^2}}\right)^{m}\Bigg\}+\sum_{m=0}^{M-1}\frac{\left(\frac{1}{2}+z\right)_m\left(\frac{1}{2}+\lambda+\mu+z\right)_m\zeta(2z+2m+1)}{m!\left(\frac{1}{2}+\lambda\right)_m}\left({-\frac{x^2}{4\pi^2}}\right)^{m}.
\end{align}
Employ \eqref{plusminus} in \eqref{withmuknueqn} to deduce that
\begin{align}\label{afteranalytic}
&2\sum_{n=1}^\infty\left(\frac{nx}{2}\right)^{z-\lambda}{}_{\mu}K_{z}(nx,\lambda)+\frac{2^{z+\mu+\lambda-1}\Gamma(z)\Gamma\left(\frac{1}{2}+\lambda+\mu\right)}{\Gamma\left(\frac{1}{2}+\lambda-z\right)}-\frac{\sqrt{\pi}2^{\mu+z+\lambda}\Gamma(\lambda+\mu)\Gamma\left(\frac{1}{2}+z\right)}{x\Gamma(\lambda-z)}\nonumber\\
&=\frac{2^{\mu+\lambda-z}}{\sqrt{\pi}}\left(\frac{x}{\pi}\right)^{2z}\frac{\Gamma\left(z+\frac{1}{2}\right)\Gamma\left(\frac{1}{2}+\lambda+\mu+z\right)}{\Gamma\left(\frac{1}{2}+\lambda\right)}\sum_{n=1}^\infty\frac{1}{n^{2z+1}}\Bigg\{\pFq{2}{1}{\frac{1}{2}+z,\frac{1}{2}+\lambda+\mu+z}{\frac{1}{2}+\lambda}{-\frac{x^2}{4\pi^2n^2}}\nonumber\\
&\qquad-\sum_{m=0}^{M-1}\frac{\left(\frac{1}{2}+z\right)_m\left(\frac{1}{2}+\lambda+\mu+z\right)_m}{m!\left(\frac{1}{2}+\lambda\right)_m}\left({-\frac{x^2}{4\pi^2n^2}}\right)^{m}\Bigg\}\nonumber\\
&\qquad+\frac{2^{\mu+\lambda-z}}{\sqrt{\pi}}\left(\frac{x}{\pi}\right)^{2z}\sum_{m=0}^{M-1}\frac{\Gamma\left(z+\frac{1}{2}+m\right)\Gamma\left(\frac{1}{2}+\lambda+\mu+z+m\right)\zeta(2z+2m+1)}{m!\Gamma\left(\frac{1}{2}+\lambda+m\right)}\left({-\frac{x^2}{4\pi^2}}\right)^{m}.
\end{align}%
By invoking \cite[Lemma~7.1]{dkk}, it is easy to see that the series on the left-hand side of \eqref{afteranalytic} is uniformly convergent Re$(z)>-M$. Also, the summand of this series is an analytic function, therefore, this series represents an analytic function of $z$ in Re$(z)>-M$ by  Weierstrass' theorem for analytic functions.

By using \eqref{2f1bound} observe that the summand of the infinite series on the right-hand side of \eqref{afteranalytic} is $O\left(n^{-2z-2m-1}\right)$ when $n$ is large. Therefore this series converges uniformly as function of $z$ in Re$(z)>-M$. Since the summand of the infinite series on the right-hand side is analytic in Re$(z)>-M$, we see by Weierstrass' theorem  that this series represents  an analytic function  of $z$  for Re$(z)>-M$. 
 Therefore, by the principle of analytic continuation, we see that \eqref{afteranalytic} holds for Re$(z)>-M$ with having limiting values on the poles at $0,-\frac{1}{2},-1,-\frac{3}{2},...,-M+\frac{1}{2}$ and Re$(x)>0$ and Re$(\mu+\lambda)>0$. This proves the theorem.
\end{proof}

Theorem \ref{nu=0} is proved next.
\begin{proof}[Theorem \textup{\ref{nu=0}}][]
Let $M=1$ in Theorem \ref{muknuwanalytic}. Let $z\to0$ on both sides of \eqref{muknuwanalyticeqn} to see that
{\allowdisplaybreaks\begin{align}\label{limitbefore}
&2\sum_{n=1}^\infty\left(\frac{nx}{2}\right)^{-\lambda}{}_{\mu}K_{0}(nx,\lambda)-\frac{\pi2^{\mu+\lambda}\Gamma(\lambda+\mu)}{x\Gamma(\lambda)}\nonumber\\
&=2^{\mu+\lambda}\frac{\Gamma\left(\frac{1}{2}+\lambda+\mu\right)}{\Gamma\left(\frac{1}{2}+\lambda\right)}\sum_{n=1}^\infty\frac{1}{n}\left\{\pFq{2}{1}{\frac{1}{2},\frac{1}{2}+\lambda+\mu}{\frac{1}{2}+\lambda}{-\frac{x^2}{4\pi^2n^2}}-1\right\}\nonumber\\
&\quad+\lim_{z\to0}\left\{\frac{2^{\mu+\lambda-z}}{\sqrt{\pi}}\left(\frac{x}{\pi}\right)^{2z}\frac{\Gamma\left(z+\frac{1}{2}\right)\Gamma\left(\frac{1}{2}+\lambda+\mu+z\right)\zeta(2z+1)}{\Gamma\left(\frac{1}{2}+\lambda\right)}-\frac{2^{z+\mu+\lambda-1}\Gamma(z)\Gamma\left(\frac{1}{2}+\lambda+\mu\right)}{\Gamma\left(\frac{1}{2}+\lambda-z\right)}\right\}.
\end{align}}
We next evaluate the limit in the above equation. We have the following well-known expansions, as $z\to0$,
{\allowdisplaybreaks\begin{align}
2^{-z}&=1-\log(2)z+O(z^2),\nonumber\\
\left(\frac{x}{2\pi}\right)^{2z}&=1+2z\log\left(\frac{x}{\pi}\right)+O(z^2),\nonumber\\
\Gamma\left(\frac{1}{2}+z\right)&=\sqrt{\pi}\left(1+\psi\left(\frac{1}{2}\right)z\right)+O(z^2),\nonumber\\
\Gamma\left(\frac{1}{2}+\lambda+\mu+z\right)&=\Gamma\left(\frac{1}{2}+\lambda+\mu\right)\left(1+\psi\left(\frac{1}{2}+\lambda+\mu\right)z\right)+O(z^2),\nonumber\\
\zeta(2z+1)&=\frac{1}{2z}+\gamma+O(z).\nonumber
\end{align}}
Also,
\begin{align}
2^z&=1+\log(2)z+O(z^2),\nonumber\\
\Gamma(z)&=\frac{1}{z}-\gamma+O(z),\nonumber\\
\frac{1}{\Gamma\left(\frac{1}{2}+\lambda-z\right)}&=\frac{1}{\Gamma\left(\frac{1}{2}+\lambda\right)}\left(1+\psi\left(\frac{1}{2}+\lambda\right)z\right)+O(z).\nonumber
\end{align}
Use all these expansion to deduce that
\begin{align}\label{limit}
&\lim_{z\to0}\left\{\frac{2^{\mu+\lambda-z}}{\sqrt{\pi}}\left(\frac{x}{\pi}\right)^{2z}\frac{\Gamma\left(z+\frac{1}{2}\right)\Gamma\left(\frac{1}{2}+\lambda+\mu+z\right)\zeta(2z+1)}{\Gamma\left(\frac{1}{2}+\lambda\right)}-\frac{2^{z+\mu+\lambda-1}\Gamma(z)\Gamma\left(\frac{1}{2}+\lambda+\mu\right)}{\Gamma\left(\frac{1}{2}+\lambda-z\right)}\right\}\nonumber\\
&=\frac{2^{\mu+\lambda-1}\Gamma\left(\frac{1}{2}+\lambda+\mu\right)}{\Gamma\left(\frac{1}{2}+\lambda\right)}\left(2\log(x)-2\log(2\pi)+3\gamma+\psi\left(\frac{1}{2}\right)+\psi\left(\frac{1}{2}+\lambda+\mu\right)-\psi\left(\frac{1}{2}+\lambda\right)\right),\nonumber\\
&=\frac{2^{\mu+\lambda-1}\Gamma\left(\frac{1}{2}+\lambda+\mu\right)}{\Gamma\left(\frac{1}{2}+\lambda\right)}\left(2\log(x)-2\log(4\pi)+2\gamma+\psi\left(\frac{1}{2}+\lambda+\mu\right)-\psi\left(\frac{1}{2}+\lambda\right)\right),
\end{align}
where we used the fact $\psi\left(\frac{1}{2}\right)=-2\log(2)-\gamma$. Now substitute the limit evaluation \eqref{limit} in \eqref{limitbefore} to get \eqref{nu=0eqn}.
\end{proof}

\begin{proof}[Corollary \textup{\ref{nu=0ofmuknu}}][]
Let $\mu=0$ in Theorem \ref{nu=0} and let $\lambda\to0$ and then use the fact ${}_0K_0(x,0)=K_0(x)$ and then use ${}_1F_0\left(\frac{1}{2};-;-x\right)=\frac{1}{\sqrt{1+x}}$ to get the result.
\end{proof}

\begin{center}
\textbf{Acknowledgements}
\end{center}
The author sincerely thanks Professor Atul Dixit for a careful reading of this article, his valuable suggestions and his support throughout this work. The 
author is an institute postdoctoral fellow at IIT Gandhinagar and sincerely thanks the institute and Professor Atul Dixit for their financial support.

\end{document}